\documentclass[preprint,review,12pt]{elsarticle}
\journal{Statistics and Probability Letters}

\usepackage[utf8]{inputenc}
\usepackage[T1]{fontenc}
\usepackage[english]{babel}
\usepackage{lmodern}
\usepackage[bookmarks={true},bookmarksopen={true}]{hyperref}
\usepackage{array}
\usepackage{amsmath}
\usepackage{amssymb}
\usepackage{amsthm}
\usepackage{enumitem}
\usepackage{mathrsfs}
\usepackage{xcolor}
\usepackage{bbm}
\usepackage{fullpage}

\DeclareFontFamily{U}{mathx}{\hyphenchar\font45}
\DeclareFontShape{U}{mathx}{m}{n}{
      <5> <6> <7> <8> <9> <10>
      <10.95> <12> <14.4> <17.28> <20.74> <24.88>
      mathx10
      }{}
\DeclareSymbolFont{mathx}{U}{mathx}{m}{n}
\DeclareFontSubstitution{U}{mathx}{m}{n}
\DeclareMathAccent{\widecheck}{0}{mathx}{"71}
\DeclareMathAccent{\wideparen}{0}{mathx}{"75}

\renewcommand{\leq}{\leqslant}
\renewcommand{\geq}{\geqslant}
\renewcommand{\epsilon}{\varepsilon}

\newcommand{\niton}{\not\ni}
\newcommand{\A}{\mathscr{A}}
\newcommand{\B}{\mathscr{B}}

\newcommand{\der}{\mathrm{d}}
\newcommand{\E}{\mathbb{E}}
\newcommand{\F}{\mathcal{E}}

\newcommand{\N}{\mathbb{N}}

\newcommand{\R}{\mathbb{R}}

\newcommand{\cara}[1]{\mathbbm{1}_{#1}}
\newcommand{\x}{\mathcal{X}}

\newcommand{\dist}{\mbox{dist}}

\newcommand{\defeq}{:=}

\newcommand{\com}[1]{\textcolor{black}{#1}}

\newcommand{\Omegaf}[1]{\com{\Omega_{F}^{#1}}}

\newcommand{\vertiii}[1]{{\left\vert\kern-0.25ex\left\vert\kern-0.25ex\left\vert #1 
    \right\vert\kern-0.25ex\right\vert\kern-0.25ex\right\vert}}

\newtheorem{prop}{Proposition}[section]
\newtheorem{lem}[prop]{Lemma}
\newtheorem{defin}[prop]{Definition}
\newtheorem{defin/prop}[prop]{Definition/Proposition}
\newtheorem{theo}[prop]{Theorem}

\begin{document}
\renewcommand{\proofname}{Proof}
\begin{frontmatter}

\title{A bound of the $\beta$-mixing coefficient for point processes in terms of their intensity functions}
\author{Arnaud Poinas\fnref{fn1}}
\ead{arnaud.poinas@univ-rennes1.fr}
\address{IRMAR, Campus de Beaulieu,  Bat. 22/23, 263 avenue du G\'en\'eral Leclerc, 35042 Rennes, France}
\date{\vspace{-5ex}}

\begin{abstract}
We prove a general inequality on $\beta$-mixing coefficients of point processes depending uniquely on their $n$-th order intensity functions. We apply this inequality in the case of determinantal point processes and show that \com{the rate of decay of the $\beta$-mixing coefficients of a wide class of DPPs} is optimal.
\end{abstract}


\end{frontmatter}

\section{Introduction}

%\doublespacing
%A common tool used in the asymptotic inference for dependent random variables is to introduce a measure of dependence between $\sigma$-algebras. Among the first ones that have been defined, we find $\alpha$-mixing and $\beta$-mixing. Since their introduction in \cite{AlphaDef} (for $\alpha$-mixing) and \cite{BetaDef} (for $\beta$-mixing), they have been used to establish moment inequalities, exponential inequalities and central limit theorems for stochastic processes (see \cite{Survey, Doukhan, Rio} for more details about mixing) with wide applications in Statistics, see for instance \cite{Bosq, Dede}.

%In particular, $\alpha$-mixing coefficients often arise in asymptotic inference of spatial point processes. For example, an $\alpha$-mixing inequality is established for Gibbs point processes 
In asymptotic inference for dependent random variables, it is necessary to quantify the dependence between $\sigma$-algebras. \com{Some of} the first measures of dependence that have been introduced are the alpha-mixing coefficients \cite{AlphaDef} and the beta-mixing coefficients \cite{BetaDef}. They have been used to establish moment inequalities, exponential inequalities and central limit theorems for stochastic processes (see \com{\cite{Survey, Nanana, Rio}} for more details about mixing) with various applications in statistics, see for instance \cite{Bosq, Dede}. In this paper, we focus on spatial point processes. As detailed below, for these models, alpha-mixing has been widely studied and exploited in the literature, but not beta-mixing in spite of its stronger properties. In a lesser extent, some alternative measures of dependence have also been used for spatial point processes, namely Brillinger mixing \cite{BiscioMix, Heinrich} (which only applies to stationary point processes \com{but has been established in \cite{Heinrich} under suitable conditions on the $\beta$-mixing coefficients}) and association \cite{Lyons, Poinas}.

The main models used in spatial point processes are Gibbs point processes, Cox processes and determinantal point processes, see \cite{Moller2} for a recent review. An $\alpha$-mixing inequality is established for Gibbs point processes in the Dobrushin uniqueness region in \cite{Follmer}. It has been used to show asymptotic normality of maximum likelihood and pseudo-likelihood estimates \cite{Jensen}. 
Similarly, some inhomogeneous Cox processes \com{like the Neyman-Scott process have also been showed to satisfy $\alpha$-mixing inequalities in \cite{Waag}}. These inequalities are at the core of asymptotic inference results in \cite{Coeur, Prok, Waag}.
Finally, an $\alpha$-mixing inequality has also been showed for determinantal point processes in \cite{Poinas} and used to get the asymptotic normality of a wide class of estimators of these models.

On the other hand, $\beta$-mixing is a stronger property than $\alpha$-mixing. It implies stronger covariance inequalities \cite{Rio} as well as a coupling theorem known as Berbee's Lemma \cite{Berbee} used in various limit theorems (for example in \cite{Example1, Example2}). Nevertheless, it rarely appears in the literature in comparison to $\alpha$-mixing. This is especially true for point processes where there has been no $\beta$-mixing property established \com{for any of the above examples. Nethertheless, $\beta$-mixing coefficients have still been used several times in random geometry and point process statistics \cite{Eins, Drei, Zwei, Vier}. In particular, it is argued in \cite{Zwei} that the $\beta$-mixing coefficient cannot be replaced by the $\alpha$-mixing coefficient when used to obtain bounds for point process characteristics related with the Palm distribution}. Our goal is to establish a general inequality for the $\beta$-mixing coefficients of a point process in terms of its intensity functions.

%Nevertheless, $\alpha$-mixing and $\beta$-mixing are not the only way of defining a measure of dependency and others have also been used in the point process literature. We can mention for example the existence of other mixing coefficients \cite{Survey, Doukhan}, Brillinger mixing \cite{BiscioMix, Heinrich} (only for stationary point processes) or even association \cite{Lyons, Poinas}.

We begin in Section \ref{Sec1} by recalling the basic definitions and properties of the $\alpha$-mixing and $\beta$-mixing coefficients and we introduce the lower sum transform which is the main technical tool that we use throughout the paper. Then, a general inequality for the $\beta$-mixing coefficients of a point process that depends only on its $n$-th order intensity functions is proved in Section \ref{Sec2}. As an example, we deduce a $\beta$-mixing inequality in the special case of determinantal point processes (DPPs) in Section \ref{Sec3} \com{whose rate of decay is optimal for a wide class of DPPs.}

\section{Preliminaries} \label{Sec1}
\subsection{Intensities of point processes}
In this paper, we consider simple point processes on $(\R^d,\com{\mathcal{B}(\R^d)},\mu)$ equipped with the euclidean norm $\|.\|$ where $d$ is a fixed integer, \com{$\mathcal{B}(\R^d)$ the Borel-$\sigma$-algebra and $\mu$ the Lebesgue measure} (more information on spatial point processes can be found in \cite{DVJ, Moller}). We denote by $\Omega$ (resp. $\Omegaf{}$) the set of locally finite (resp. finite) point configurations in $\R^d$. For all functions $f:\Omegaf{}\rightarrow\R$, $n\in\N$ and $x=(x_1,\cdots,x_n)\in(\R^d)^n$, we write $f(x)$ for $f(\{x_1,\cdots,x_n\})$ by an abuse of notation. Finally, we write $|A|$ for the cardinal of a finite set $A$ and $\|f\|_{\infty}$ for the uniform norm of a function $f$.

We begin by recalling that the $n$-th order intensity functions \com{(also called $n$-th order product density)} are defined the following way (see~\cite{Moller}).
\begin{defin} \label{def:rho}
Let $X$ be a simple point process on $\R^d$ and $n\geq 1$ be an integer. If there exists a non negative function $\rho_n:(\R^d)^n\rightarrow\R$ such that
\begin{equation} \label{defrho}
\E\left [ \sum_{x_1,\cdots,x_n\in X}^{\neq}f(x_1,\cdots,x_n)\right ]=\int_{(\R^d)^n} f(x)\rho_n(x)\der\mu^n(x).
\end{equation}
for all locally integrable functions $f:(\R^d)^n\rightarrow\R$ then $\rho_n$ is called the $n$th order intensity function of $X$.
\end{defin}
In the rest of the paper, all point processes will be considered to admit bounded $n$-th order intensity function for all $n\geq 1$.

\subsection{Mixing}
%In this part, we recall generalities about mixing in the general case of two sub $\sigma$-algebras and in the more specific case of point processes.\\
\medskip
Consider a probability space $(\mathcal{X},\mathcal{F},\mathbb{P})$ and $\A,\B$ two sub $\sigma$-algebras of $\mathcal{F}$. Let $\mathbb{P}_\A$ and $\mathbb{P}_\B$ be the respective restrictions of $\mathbb{P}$ to $\A$ and $\B$ and define the probability $\mathbb{P}_{\A\otimes\B}$ on the product $\sigma$-algebra by $\mathbb{P}_{\A\otimes\B}(A\times B)=\mathbb{P}(A\cap B)$ for all $A\in\A$ and $B\in\B$. The $\alpha$-mixing and $\beta$-mixing coefficients (also called strong-mixing and absolute regularity coefficients) are defined as the following measures of dependence between $\A$ and $\B$ \cite{Nanana, Rio}:
\begin{align}
\alpha(\A,\B)&\defeq\sup\{|\mathbb{P}(A\cap B)-\mathbb{P}(A)\mathbb{P}(B)|: A\in\A, B\in\B\},\\
\beta(\A,\B)&\defeq\|\mathbb{P}_{\A\otimes\B}-\mathbb{P}_{\A}\otimes\mathbb{P}_{\B}\|_{TV},\label{beta1}
\end{align}
where $\|.\|_{TV}$ is the total variation of a signed norm.\\
\medskip
For a given point process $X$ and a bounded set $A\subset\R^d$, we denote by $\mu(A)\defeq\int_A \der\mu(x)$ the volume of $A$ and $\F(A)$ the $\sigma$-algebra generated by $X\cap A$. Finally, for all  $A,B\subset\R^d$, we write $\dist(A,B)$ for the infimum of $\|y-x\|$ where $(x,y)\in A\times B$. The $\beta$-mixing coefficients of the point process $X$ are then defined by
$$\beta_{p,q}(r)\defeq\sup\{\beta(\F(A),\F(B)) : \mu(A)\leq p, \mu(B)\leq q, \mbox{dist}(A,B)>r\},$$
and we say that the point process $X$ is beta-mixing if $\beta_{p,q}(r)$ vanishes when $r\rightarrow +\infty$ for all $p,q>0$. \com{The $\alpha$-mixing coefficients can be defined in a similar way.}

Our goal is to prove that under appropriate assumptions over the intensity functions $\rho_n$ of $X$ we have a $\beta$-mixing property.

\subsection{Lower sum transform}
The main tool we use throughout this paper is the so-called lower sum operator (see \cite{Comb}). Notice that when $f$ is a symmetric function the term in the expectation in \eqref{defrho} can be written as $n!\sum_{Z\subset X}f(Z)\cara{|Z|=n}$. This motivates the following definition:
\begin{defin} \label{def:mobius}
Let $f$ be a real function defined over $\Omegaf{}$. The lower sum of $f$ is the linear operator defined by
\begin{equation} \label{mob1}
\hat{f}:X\mapsto\sum_{Z\subset X}f(Z).
\end{equation}
\end{defin}
As shown in Example 4.19 in \cite{Comb}, this operator admits the following inverse transform.
\begin{prop}[{{\cite[Theorem~4.18]{Comb}}}]
The operator \eqref{mob1} admits an inverse transform $\check{f}$, called the lower difference of $f$, defined by
\begin{equation} \label{mob2}
\check{f}:X\mapsto\sum_{Z\subset X}(-1)^{|X\backslash Z|}f(Z).
\end{equation}
\end{prop}
These definitions extend to functions over $\Omegaf{2}$ by defining
$$\hat{f}:(X_1,X_2)\mapsto\hspace{-0.5cm}\sum_{Z_1\subset X_1, Z_2\subset X_2}\hspace{-0.5cm}f(Z_1,Z_2)~~\mbox{and}~~\check{f}:(X_1,X_2)\mapsto\hspace{-0.5cm}\sum_{Z_1\subset X_1, Z_2\subset X_2}\hspace{-0.5cm}(-1)^{|X_1\backslash Z_1|+|X_2\backslash Z_2|}f(Z_1,Z_2).
$$
\com{In a similar way, we could also extend these definitions to $\Omegaf{n}$ for any $n$ but we will only need the case $n\leq 2$ for the remaining of the paper.} These operators allow us to give an explicit expression for the expectation of a functional of a point process with respect to its intensity functions.
\begin{prop}
If $X$ is an almost surely finite point process such that $\E[4^{|X|}]<+\infty$, then
\begin{equation} \label{MainExpectation}
\E[f(X)]=\sum_{n=0}^{+\infty}\frac{1}{n!}\int_{(\R^d)^n}\check{f}(x)\rho_n(x)\der\mu^n(y)
\end{equation}
for all bounded functions $f:\Omegaf{}\rightarrow\R$. Moreover, if $X'$ is a point process independent from $X$ satisfying the same assumptions than $X$ and with $n$-th order intensity functions $\rho_n'$, then
\begin{equation} \label{MainExpectation2}
\E[f(X,X')]=\sum_{m,n=0}^{+\infty}\frac{1}{m!n!}\int_{(\R^d)^{m+n}}\check{f}(x,y)\rho_m(x)\rho'_n(y)\der\mu^m(x)\der\mu^n(x)
\end{equation}
for all bounded functions $f:\Omegaf{2}\rightarrow\R$.
\end{prop}
\begin{proof}
Using the bound $|\check f(x)|\leq \com{\|f\|_{\infty}\mbox{card}\{Z,Z\subset X\}=\|f\|_{\infty}2^{|x|}}$ we get
\begin{equation} \label{justif}
\sum_{n\geq 0}\E\left[\left|\sum_{\substack{Z\subset X \\ |Z|=n}}\check f(Z)\right|\right]\leq \sum_{n\geq 0}\E\left[2^{|X|}\binom{|X|}{n}\right]\|f\|_{\infty}=\|f\|_{\infty}\E\left[4^{|X|}\right]<+\infty.
\end{equation}
Since we can write 
$$f(X)=\hat{\check{f}}(X)=\sum_{Z\subset X}\check f(Z)=\sum_{n\geq 0}\sum_{\substack{Z\subset X \\ |Z|=n}}\check f(Z)~~\mbox{a.s.},$$
then
$$\E[f(X)]=\sum_{n\geq 0}\E\left[\sum_{\substack{Z\subset X \\ |Z|=n}}\check f(Z)\right]=\sum_{n=0}^{+\infty}\frac{1}{n!}\int_{(\R^d)^{n}}\check{f}(x)\rho_n(x)\der\mu^n(x)$$
\com{where the inversion of the first sum and the expectation is a consequence of \eqref{justif}.} Similarly, for all functions $f:\Omegaf{2}\rightarrow\R$ we have
\begin{align*}
\E[f(X,X')]&=\E[\E[f(X,X')|X']]\\
&=\E\left [\sum_{m=0}^{+\infty}\frac{1}{m!}\int_{(\R^d)^m}\left(\sum_{z\subset x}(-1)^{m-|z|}f(z,X')\right)\rho_m(x)\der\mu^m(x)\right ]\\
&=\sum_{m,n=0}^{+\infty}\frac{1}{m!n!}\int_{(\R^d)^{m+n}}\check{f}(x,y)\rho_m(x)\rho'_n(y)\der\mu^m(x)\der\com{\mu^n(y)},
\end{align*}
where all inversions of expectation with sum and integrals can be justified in a similar way than \eqref{justif}.
\end{proof}

\section{\texorpdfstring{$\beta$}{Beta}-mixing of point processes with known intensity functions} \label{Sec2}
Our main result is the following inequality showing that if all $\rho_m(x)\rho_n(y)-\rho_{m+n}(x,y)$ vanish fast enough when $\|y-x\|\rightarrow +\infty$ for all $m,n\in\N$, then the underlying point process is $\beta$-mixing.
\begin{theo} \label{main}
Let $X$ be a simple point process on $(\R^d,\mu)$ such that $\E[4^{|X\cap A|}]<+\infty$ for all bounded subsets $A\subset\R^d$. Then, for all $p,q,r\in\R_+$,
\begin{equation} \label{fin}
\beta_{p,q}(r)\leq\hspace{-0.5cm}\sup_{\substack{\mu(A)< p, \mu(B)< q \\ \rm{dist} (A,B)>r}}\left(\sum_{m,n=0}^{+\infty}\frac{\com{2^{n+m-1}}}{m!n!}\int_{A^m\times B^n}|\rho_m(x)\rho_n(y)-\rho_{m+n}(x,y)|\der\mu^m(x)\der\mu^n(y)\right).
\end{equation}
\end{theo}
Before giving the proof of Theorem \ref{main}, we need the following lemmas showing the behaviour of $f(X\cap A,X\cap B)$ and $f(X\cap A,X'\cap B)$ under the lower difference operator.
\begin{lem} \label{lemnul}
Let $A\subset\R^d$, $f:\Omegaf{}\rightarrow\R$ \com{and define $f_A:X\mapsto f(X\cap A)$}. Then,
$$\widecheck{f_A}(X)=\check{f}(X)\cara{X\subset A}.$$
\end{lem}
\begin{proof}
If $X\subset A$ then the result is trivial. Otherwise, there exists $x\in X\backslash A$ and we can write
\begin{align*}
\widecheck{f_A}(X)&=\sum_{Z\subset X, Z\ni x}(-1)^{|X\backslash Z|}f(Z\cap A)+\sum_{Z\subset X, Z\niton x}(-1)^{|X\backslash Z|}f(Z\cap A)\\
&\com{=\sum_{Z\subset X, Z\niton x}(-1)^{|X\backslash Z|-1}f\big((Z\cup\{x\})\cap A\big)+\sum_{Z\subset X, Z\niton x}(-1)^{|X\backslash Z|}f(Z\cap A)}\\
&=\sum_{Z\subset X, Z\niton x}(-1)^{|X\backslash Z|-1}f(Z\cap A)+\sum_{Z\subset X, Z\niton x}(-1)^{|X\backslash Z|}f(Z\cap A)\com{~=0}.
\end{align*}
\end{proof}
This result can be extended to multivariate functions: The lower difference of $(X_1,X_2)\rightarrow f(X_1\cap A_1,X_2\cap A_2)$ is $\check{f}(X_1,X_2)\cara{\{X_1\subset A_1\}}\cara{\{X_2\subset A_2\}}$.

\begin{lem} \label{lemnul2}
For all $f:\Omegaf{2}\rightarrow\R$ and $A,B$ disjoint subsets of $\R^d$, let us define the function $g:X\mapsto f(X\cap A,X\cap B)$. The lower difference of $g$ satisfies
$$\check{g}(X)=\check{f}(X\cap A,X\cap B)\cara{\{X\subset A\cup B\}}.$$
\end{lem}
\begin{proof}
Using Lemma~\ref{lemnul} we get that $\check{g}(X)=0$ whenever $X$ is not a subset of $A\cup B$. Otherwise, since $A$ and $B$ are disjoint sets,
$$\check{g}(X)=\sum_{Z\subset X}(-1)^{|X\backslash Z|}f(Z\cap A,Z\cap B)=\sum_{\substack{U\subset X\cap A \\ V\subset X\cap B}}(-1)^{|(X\cap A)\backslash U|+|(X\cap B)\backslash V|}f(U,V)$$
which, by definition, is equal to $\check{f}(X\cap A,X\cap B)$.
\end{proof}

We now have the necessary tools required for the proof of Theorem \ref{main}.

\begin{proof}[Proof of Theorem \ref{main}]
Let $p,q>0$ and $A,B$ be two disjoint subsets of $\R^d$ such that $\mu(A)\leq p$ and $\mu(B)\leq q$. \com{Using one of the characterizations of the total variation distance, the $\beta$-mixing coefficient between $\F(A)$ and $\F(B)$ can be expressed as
$$\beta(\F(A),\F(B))=\frac{1}{2}\sup_{\|f\|_{\infty}=1}\left|E[f(X\cap A,X\cap B)]-E[f(X\cap A,X'\cap B)]\right|$$
where $X'$ is an independent copy of $X$.} Since $X\cap A$, $X'\cap B$ and $X\cap B$ are finite a.s. we can apply \eqref{MainExpectation2} which, combined with Lemma \ref{lemnul}, gives us
\begin{equation}\label{reve1}
\E[f(X\cap A,X'\cap B)]=\sum_{m,n=0}^{+\infty}\frac{1}{m!n!}\int_{A^m\times B^n}\check{f}(x,y)\rho_m(x)\rho_n(y)\der\mu^m(x)\der\mu^n(y).
\end{equation}
On the other hand, by combining \eqref{MainExpectation} with Lemma~\ref{lemnul2}, we get
\begin{align*}
\E[f(X\cap A,X\cap B)]=\sum_{n=0}^{+\infty}\frac{1}{n!}\int_{(A\cup B)^n}\check{f}(x\cap A,x\cap B)\rho_n(x)\der\mu^n(x).
\end{align*}
Since $A$ and $B$ are disjoint sets and by symmetry of $\check{f}(x\cap A,x\cap B)\rho_n(x)$, we can simplify the above expression into
\begin{align}
\E[f(X\cap A,X\cap B)]&=\sum_{n=0}^{+\infty}\sum_{m=0}^n\frac{1}{n!}\binom{n}{m}\int_{A^m\times B^{n-m}}\check{f}(x,y)\rho_n(x,y)\der\mu^m(x)\der\mu^{n-m}(y)\nonumber\\
&=\sum_{m,n=0}^{+\infty}\frac{1}{m!n!}\int_{A^m\times B^n}\check{f}(x,y)\rho_{m+n}(x,y)\der\mu^m(x)\der\mu^n(y).\label{reve2}
\end{align}
Combining \com{\eqref{reve1} and \eqref{reve2}} yields that $|\E[f(X\cap A,X\cap B)]-\E[f(X\cap A,X'\cap B)]|$ is equal to
$$\left |\sum_{m,n=0}^{+\infty}\frac{1}{m!n!}\int_{A^m\times B^n}\check{f}(x,y)(\rho_{m}(x)\rho_n(y)-\rho_{m+n}(x,y))\der\mu^m(x)\der\mu^n(y) \right |$$
which is bounded by
$$\sum_{m,n=0}^{+\infty}\frac{2^{n+m}}{m!n!}\int_{A^m\times B^n}|\rho_m(x)\rho_n(y)-\rho_{m+n}(x,y)|\der\mu^m(x)\der\mu^n(y)$$
\com{when $\|f\|_\infty=1$} and where we used the bound $|\check{f}(x,y)|\leq 2^{|x|+|y|}$.
\end{proof}

\section{Application to determinantal point processes} \label{Sec3}
We can directly apply Theorem \ref{main} to determinantal point processes. First introduced in \cite{Macchi} \com{under its current form} to model fermion systems, DPPs are a broad class of repulsive point processes. We recall that a DPP $X$ with kernel $K:(\R^d)^2\rightarrow\R$ is defined by its intensity functions
$$\rho_n(x_1,\cdots,x_n)=\det(K[x])~~~~\forall x\in (\R^d)^n,~\forall n\in\N$$
where we denote by $K[x]$ the matrix $(K(x_i,x_j))_{1\leq i,j\leq n}$. \com{Existence and uniqueness conditions as well as general information on DPPs can be found in \cite{Hough}.} The application of Theorem \ref{main} to DPPs gives us the following $\beta$-mixing condition:
\begin{theo}
Let $X$ be a DPP with kernel $K$ and define
$$\omega(r)\defeq\sup_{\|y-x\|\geq r}|K(x,y)|.$$
If $K$ is bounded and $\omega(r)\underset{r\rightarrow +\infty}{\longrightarrow} 0$ then $X$ is $\beta$-mixing. In particular,
$$\beta_{p,q}(r)\leq \com{2}pq(1+2p\|K\|_{\infty})(1+2q\|K\|_{\infty})e^{2\|K\|_{\infty}(p+q)}\omega(r)^2.$$
%Furthermore, if there exists $N\in\N$ such that $\rank(\mathcal{K})\leq N$ then
%$$\beta_{p,q}(r)\leq 4pqN^29^N\omega(r)^2.$$
\end{theo}
\com{Unfortunately, this result does not give a bound for $\beta_{p,\infty}(r)$ which yet is necessary in almost all limit theorems based on beta-mixing.}
\begin{proof}
Since $\E[4^{|X\cap A|}]<+\infty$ for all bounded sets $A$ (see \cite[Lemma B.5]{Poinas}) then the $\beta$-mixing coefficients of $X$ satisfy \eqref{fin} by Theorem \ref{main}. Let $x=(x_1,\cdots,x_n)$ and $y=(y_1,\cdots,y_m)$, we need to control $|\det(K[x])\det(K[y])-\det(K[x,y])|$ where $\|x-y\|\geq r$. By \cite[Lemma B.4]{Poinas}, we get the bound
$$0\leq\det(K[x])\det(K[y])-\det(K[x,y])\leq nm\|K\|_{\infty}^{n+m-2}\sum_{i=1}^n\sum_{j=1}^m K(x_i,y_j)^2.$$
Injecting this bound into \eqref{fin} gives us
\begin{align}
\beta_{p,q}(r)&\leq\sum_{n,m=0}^{+\infty}\frac{n^2m^2\com{2^{n+m-1}}p^{n-1}q^{m-1}\|K\|_{\infty}^{n+m-2}}{n!m!}\sup_{\substack{|A|< p, |B|< q \\ \rm{dist} (A,B)>r}}\int_{A\times B}|K(x,y)|^2\der \mu(x)\der \mu(y)\label{last}\\
&\leq\sum_{n,m=0}^{+\infty}\frac{n^2m^2\com{2^{n+m-1}}p^nq^m\|K\|_{\infty}^{n+m-2}}{n!m!}\omega(r)^2\nonumber\\
&=\com{2}pq(1+2p\|K\|_{\infty})(1+2q\|K\|_{\infty})e^{2(p+q)\|K\|_{\infty}}\omega(r)^2.\nonumber
\end{align}
In particular, if $\omega(r)$ vanishes when $r\rightarrow +\infty$ then $X$ is $\beta$-mixing.\end{proof}

In conclusion, the $\beta$-mixing coefficients of DPPs decay at the same rate than $|K(x,y)|^2$ does when $x$ and $y$ deviate from each other. For example, kernels of the Ginibre ensemble or the Gaussian unitary ensemble have an exponential decay (see \cite{Hough}). Moreover, among translation-invariant kernels used in spatial statistics (see \cite{Biscio, Lavancier}), all kernels of the Laguerre-Gaussian family also have an exponential decay while kernels of the Whittle-Matérn and Cauchy family satisfy $\omega(r)=o(r^{-d})$ and kernels of the Bessel family satisfy $\omega(r)=o(r^{-(d+1)/2})$.

It is also worth noticing that Theorem \ref{main} is optimal in the sense that for a wide class of DPPs, the $\beta$-mixing coefficients $\beta_{p,q}(r)$ do not decay faster\com{, when $r$ goes to infinity, than the supremum of $\int_{A\times B}|K(x,y)|^2\der\mu(x)\der\mu(y)$ for all $A,B$ such that $\mu(A)\leq p$, $\mu(B)\leq q$ and $\dist(A,B)\geq r$} as stated in the following proposition.
\begin{prop}
Let $X$ be a DPP with a non-negative bounded kernel $K$ such that the eigenvalues of \com{its associated integral operator} are all in $[0,M]$ where $M<1$. Then, for all $p,q,r>0$,
\begin{multline*}
2(1-M)^{\frac{(p+q)\|K\|_{\infty}}{M}}\sup_{\substack{\com{\mu(A)< p, \mu(B)< q} \\ \rm{dist} (A,B)>r}}\int_{A\times B}|K(x,y)|^2\com{\der\mu(x)\der\mu(y)}\leq\beta_{p,q}(r)\\
\leq \com{2}(1+2p\|K\|_{\infty})(1+2q\|K\|_{\infty})e^{2(p+q)\|K\|_{\infty}}\sup_{\substack{\com{\mu(A)< p, \mu(B)< q} \\ \rm{dist} (A,B)>r}}\int_{A\times B}|K(x,y)|^2\com{\der\mu(x)\der\mu(y)}.
\end{multline*}
\end{prop}
\begin{proof}
The first inequality is a consequence of the fact that $\beta_{p,q}(r)\geq 2\alpha_{p,q}(r)$ and \cite[Proposition 4.3]{Poinas}. The second inequality is equivalent to \eqref{last} once the sum has been developed.
\end{proof}

\section*{Acknowledgement}
The author would like to thank Bernard Delyon for bringing the problem to his attention as well as suggesting the use of the lower sum operator. The author would also like to thank Frédéric Lavancier for his suggestions and corrections during the writing of this paper.

\section*{References}

\nocite{*}
\bibliographystyle{elsarticle-num-names}
\bibliography{ref}

\end{document}